\documentclass[10pt]{amsart}

\topskip  \usepackage{amsmath,amsthm,cite}
\usepackage{amssymb}

\begin{document}
\newtheorem{theorem}{Theorem}
\newtheorem{lemma}[theorem]{Lemma}
\newtheorem{proposition}[theorem]{Proposition}
\newtheorem{corollary}[theorem]{Corollary}

\title{Nonlinear differential equation for Korobov numbers}
\author[D.S.~Kim]{Dae San Kim}
\address{Department of Mathematics, Sogang University, Seoul 121-742, South Korea}
\email{dskim@sogang.ac.kr}
\author[T.~Kim]{Taekyun Kim}
\address{Department of Mathematics, Kwangwoon University, Seoul 139-701,  South Korea}
\email{tkkim@kw.ac.kr}
\author[H.I.~Kwon]{Hyuck-In Kwon}
\address{Department of Mathematics, Kwangwoon University, Seoul 139-701, South Korea}
\email{sura@kw.ac.kr}
\author[T.~Mansour]{Toufik Mansour}
\address{Department of Mathematics, University of Haifa, 3498838 Haifa, Israel}
\email{tmansour@univ.haifa.ac.il}

\begin{abstract}
In this paper, we present nonlinear differential equations for
the generating functions for the Korobov numbers and for the
Frobenuius-Euler numbers. As an application, we find an explicit
expression for the $n$th derivative of $1/\log(1+t)$.
\bigskip

\noindent{\bf Keywords}: Korobov numbers; Frobenuius-Euler numbers
\end{abstract}
\maketitle

\section{Introduction}
The {\em Krobov polynomials} $K_n(\lambda,x)$ of the first kind are given by
$$\frac{\lambda t}{(1+t)^\lambda-1}(1+t)^x=\sum_{n\geq0}K_n(\lambda,x)\frac{t^n}{n!}.$$
For example,
\begin{align*}
K_0(\lambda,x)&=1,\\
K_1(\lambda,x)&=\frac{1}{2}(2x-\lambda+1),\\
K_2(\lambda,x)&=\frac{1}{12}(6x^2-1+\lambda^2-6\lambda x),\\
K_3(\lambda,x)&=\frac{1}{24}(2\lambda
x-2x^2-\lambda+2x+1)(1-2x+\lambda).
\end{align*}
When $x=0$, $K_n(\lambda)=K_n(\lambda,0)$ are called the {\em
Korobov numbers of the first kind} or just {\em Korobov numbers}.
Since 2002, Korobov polynomials and numbers have been received a
lot of attention (see \cite{Kor1,KD1}). In particular, these
polynomials are used to derive some interpolation formulas of many
variables and a discrete analog of the Euler summation formula
(see \cite{Ust}). The {\em Frobenius-Euler numbers} $H_n(\mu)$ are
defined by the generating function (see \cite{A01,KM01,Y1,H1,R1})
$$\left(\frac{1-\mu}{e^t-\mu}\right)
=\sum_{n\geq0}H_n(\mu)\frac{t^n}{n!},\quad \mu \neq1.$$ Recently,
the degenerate Bernoulli and Euler polynomials related to Korobov
polynomials are studied by several authors
(\cite{C1,C2,KK1,KK2,K1,K2,K3,K4,KKS,KM01}) and Kim and Kim-Kim
derived some interesting identities of Frobenius-Euler polynomials
and the Bernoulli polynomials of the second kind arising from
nonlinear differential equations (see \cite{K1,K2,KK1}).

The main goal of this paper is to write a nonlinear differential
equation satisfying the generating function
$\frac{\lambda t}{(1+t)^\lambda-1}$ for Korobov numbers
$K_n(\lambda)$, and a nonlinear differential equation
satisfying the generating function
$\left(\frac{1-\mu}{e^t-\mu}\right)$ for Frobenius-Euler numbers
$H_n(\mu)$, see next sections. Also, we present in each case some
applications for our nonlinear differential equations. For
instance, we find an explicit expression for the $n$th derivative
of $1/\log(1+t)$, see Corollary \ref{coA01}.

\section{Korobov numbers}
Put $F=F(t)=F(t;\lambda)=\frac{1}{(1+t)^\lambda-1}$
($\lambda\neq0$). By differentiating respect to $t$, we have
\begin{align}
F^{(1)}&=\frac{-1}{((1+t)^\lambda-1)^2}\cdot\frac{\lambda(1+t)^\lambda}{1+t}
=\frac{-\lambda}{1+t}\cdot\frac{(1+t)^\lambda-1+1}{((1+t)^\lambda-1)^2}
=\frac{-\lambda}{1+t}(F+F^2).\label{eqA1}
\end{align}
Now, we let
$$F^{(N)}=\frac{(-1)^N\lambda}{(1+t)^N}\sum_{i=1}^{N+1}a_{i-1}(N;\lambda)F^i
=\frac{(-1)^N\lambda}{(1+t)^N}\sum_{i=1}^{N+1}a_{i-1}(N)F^i,$$ for
all $N\geq0$, and $a_i(N)=0$ for all $i\geq N+1$. Note that
$F=F^{(0)}=\lambda a_0(0)F$, which implies that
$a_0(0)=\frac{1}{\lambda}$. Also, by \eqref{eqA1}, we have
$a_0(1)=a_1(1)=1$. For $N+1$, we have
\begin{align*}
F^{(N+1)}&=\frac{d}{dt}\left(\frac{(-1)^{N}\lambda}{(1+t)^{N}}\sum_{i=1}^{N+1}a_{i-1}(N)F^i\right)\\
&=\frac{(-1)^{N+1}\lambda
N}{(1+t)^{N+1}}\sum_{i=1}^{N+1}a_{i-1}(N)F^i
+\frac{(-1)^N\lambda}{(1+t)^N}\sum_{i=1}^{N+1}ia_{i-1}(N)F^{i-1}F^{(1)},
\end{align*}
which, by \eqref{eqA1}, gives
\begin{align*}
F^{(N+1)}&=\frac{(-1)^{N+1}\lambda
N}{(1+t)^{N+1}}\sum_{i=1}^{N+1}a_{i-1}(N)F^i
+\frac{(-1)^N\lambda}{(1+t)^N}\sum_{i=1}^{N+1}ia_{i-1}(N)F^{i-1}\frac{-\lambda(F+F^2)}{1+t}\\
&=\frac{(-1)^{N+1}\lambda}{(1+t)^{N+1}}\left(\sum_{i=1}^{N+1}Na_{i-1}(N)F^i
+\sum_{i=1}^{N+1}\lambda ia_{i-1}(N)F^i+\sum_{i=2}^{N+2}\lambda(i-1)a_{i-2}(N)F^{i}\right)\\
\end{align*}
By assumption,
$F^{(N+1)}=\frac{(-1)^{N+1}\lambda}{(1+t)^{N+1}}\sum_{i=1}^{N+2}a_{i-1}(N+1)F^i$,
and, by comparing the coefficients of $F^i$ on both sides, we
obtain the following recurrence relation
\begin{align}
a_{i-1}(N+1)=(N+i\lambda)a_{i-1}(N)+\lambda(i-1)a_{i-2}(N),\quad
i=2,\ldots,N+1\label{eqA2}
\end{align}
with $a_j(N+1)=0$ whenever $j\geq N+2$,
\begin{align}
a_0(N+1)=(N+\lambda)a_0(N),\quad
a_{N+1}(N+1)=\lambda(N+1)a_N(N).
\end{align}

Recalling that $a_0(0)=\frac{1}{\lambda}$ and $a_0(1)=a_1(1)=1$, by induction on $N$, we
have
\begin{align*}
a_0(N+1)&=(N+\lambda)(N-1+\lambda)\cdots(1+\lambda)=(N+\lambda)_N,\\
a_{N+1}(N+1)&=a_1(1)\prod_{j=2}^{N+1}(j\lambda)=\lambda^N(N+1)!.
\end{align*}
In next lemma, we treat the general case.

\begin{lemma}\label{lemA1}
The coefficients $a_j(N)$, $j=1,2,\ldots,N$, satisfy
$$a_j(N)=j\lambda\sum_{i=0}^{N-j}(N+(j+1)\lambda-1)_ia_{j-1}(N-i-1).$$
\end{lemma}
\begin{proof}
By \eqref{eqA2}, we have
\begin{align*}
a_j(N+1)&=j\lambda a_{j-1}(N)+(N+(j+1)\lambda)a_j(N) \\
&=j\lambda a_{j-1}(N)
+j\lambda(N+(j+1)\lambda)a_{j-1}(N-1)\\
&+(N+(j+1)\lambda)(N+(j+1)\lambda-1)a_j(N-1).
\end{align*}
By induction and the initial condition $a_j(j)=\lambda^{j-1}j!$,
we obtain
$$a_j(N+1)=j\lambda\sum_{i=0}^{N+1-j}(N+(j+1)\lambda)_ia_{j-1}(N-i),$$
as required.
\end{proof}

By Lemma \ref{lemA1}, we can state the following result.
\begin{theorem}\label{thA1}
The function $F=F(t)=\frac{1}{(1+t)^\lambda-1}$ with
$\lambda\neq0$ satisfies the nonlinear differential equation
$$F^{(N)}=\frac{(-1)^N\lambda}{(1+t)^N}\sum_{i=1}^{N+1}a_{i-1}(N)F^i,$$
where $a_0(N)=(N+\lambda-1)_{N-1}$ with
$a_0(0)=\frac{1}{\lambda}$, and
$$a_j(N)=j\lambda\sum_{i=0}^{N-j}(N+(j+1)\lambda-1)_ia_{j-1}(N-i-1),$$
for $j=1,2,\ldots,N$.
\end{theorem}

As an application for Theorem \ref{thA1}, let us consider the
limit $\lim_{\lambda\rightarrow0}\lambda F(t)$. By the fact that
$\lim_{\lambda\rightarrow0}\lambda F(t)=\frac{1}{\log(1+t)}$,
we obtain
$$\frac{d^N}{dt^N}\frac{1}{\log(1+t)}
=\frac{(-1)^N}{(1+t)^N}\sum_{i=2}^{N+1}\lim_{\lambda\rightarrow0}\lambda^{2-i}a_{i-1}(N;\lambda)
\frac{1}{\log^i(1+t)}.$$ Combining with the previous result in
\cite{KK1}, we obtain
$$\lim_{\lambda\rightarrow0}\lambda^{2-i}a_{i-1}(N;\lambda)=(i-1)!(N-1)!H_{N-1,i-2},$$
where $2\leq i\leq N+1$ and $H_{N,j}$ is given by
$$H_{N,j}=\left\{\begin{array}{ll}
1,& j=0,\\
H_N=\sum_{i=1}^N\frac{1}{i},&j=1,\\
\sum_{i=1}^N\frac{H_{i-1,j-1}}{i},&2\leq j\leq
N,\end{array}\right.$$ with $H_{0,j-1}=0$ when $j\geq2$. Hence, by
Theorem \ref{thA1}, we can state the following corollary.
\begin{corollary}\label{coA01}
For all $N\geq1$,
$$\frac{d^N}{dt^N}\frac{1}{\log(1+t)}
=\frac{(-1)^N(N-1)!}{(1+t)^N}\sum_{i=2}^{N+1}\frac{(i-1)!H_{N-1,i-2}}{\log^i(1+t)}.$$
\end{corollary}
For example,
\begin{align*}
\frac{d}{dt}\frac{1}{\log(1+t)}&=\frac{-1}{1+t}\frac{1}{\log^2(1+t)},\\
\frac{d^2}{dt^2}\frac{1}{\log(1+t)}&=\frac{1}{(1+t)^2}\left(\frac{1}{\log^2(1+t)}+\frac{2}{\log^3(1+t)}\right),\\
\frac{d^3}{dt^3}\frac{1}{\log(1+t)}&=\frac{-1}{(1+t)^3}\left(\frac{2}{\log^2(1+t)}+\frac{6}{\log^3(1+t)}+\frac{6}{\log^4(1+t)}\right).
\end{align*}

As another application, let us consider the generating function
for the Korobov numbers
$$\frac{\lambda
t}{(1+t)^\lambda-1}=\sum_{n\geq0}K_n(\lambda)\frac{t^n}{n!}.$$
More generally, the Korobov numbers of order $m$ are defined via the
following generating function
$$\left(\frac{\lambda
t}{(1+t)^\lambda-1}\right)^m=\sum_{n\geq0}K_n^{(m)}(\lambda)\frac{t^n}{n!}.$$

\begin{theorem}
For all $N\geq1$,
\begin{align*}
&\sum_{i=0}^{\min(n,N)}\lambda^{i-N+1}(n)_ia_{N-i}(N)K_{n-i}^{(N+1-i)}(\lambda)\\
&\qquad\qquad\qquad=\left\{\begin{array}{ll} N!(N)_n,&0\leq n\leq N,\\
(-1)^N\sum_{\ell=0}^{n-N-1}\binom{N}{\ell}\frac{K_{n-\ell}(\lambda)}{n-\ell}(n)_{N+1+\ell},&
n\geq N+1.
\end{array}\right.
\end{align*}
\end{theorem}
\begin{proof}
Note that
$$F=\frac{1}{(1+t)^\lambda-1}=\frac{1}{\lambda}\left(\frac1t+
\sum_{n\geq0}K_{n+1}(\lambda)\frac{t^{n}}{(n+1)!}\right).$$ Thus,
\begin{align*}
F^{(N)}=\frac{1}{\lambda}\left((-1)^NN!t^{-N-1}+ \sum_{n\geq
N}K_{n+1}(\lambda)(n)_N\frac{t^{n-N}}{(n+1)!}\right),
\end{align*}
which implies
\begin{align*}
t^{N+1}F^{(N)}=\frac{1}{\lambda}\left((-1)^NN!+ \sum_{n\geq
N}K_{n+1}(\lambda)(n)_N\frac{t^{n+1}}{(n+1)!}\right).
\end{align*}
Multiplying both sides by $(1+t)^N$, we obtain
\begin{align}
&(1+t)^Nt^{N+1}F^{(N)}\notag\\
&=\frac{1}{\lambda}\left((-1)^NN!\sum_{n=0}^N(N)_n\frac{t^n}{n!}+
\sum_{\ell=0}^N(N)_\ell\frac{t^\ell}{\ell!}\sum_{n\geq
N}K_{n+1}(\lambda)(n)_N\frac{t^{n+1}}{(n+1)!}\right)\notag\\
&=\frac{1}{\lambda}\left((-1)^NN!\sum_{n=0}^N(N)_n\frac{t^n}{n!}+
\sum_{n\geq
N+1}\sum_{\ell=0}^{n-N-1}\binom{N}{\ell}\frac{K_{n-\ell}(\lambda)}{n-\ell}(n)_{N+1+\ell}\frac{t^n}{n!}\right).\label{eqAA1}
\end{align}
On the other hand, by Theorem \ref{thA1}, we have
\begin{align*}
(1+t)^Nt^{N+1}F^{(N)}&=(-1)^N\lambda\sum_{i=1}^{N+1}a_{i-1}(N)\frac{t^{N+1}}{((1+t)^\lambda-1)^i}\\
&=(-1)^N\lambda\sum_{i=1}^{N+1}\lambda^{-i}a_{i-1}(N)\left(\frac{\lambda t}{(1+t)^\lambda-1}\right)^i t^{N+1-i}\\
&=(-1)^N\lambda\sum_{i=0}^{N}\lambda^{i-N-1}a_{N-i}(N)\left(\frac{\lambda
t}{(1+t)^\lambda-1}\right)^{N+1-i}t^{i}.
\end{align*}
Thus by the generating function for the Korobov numbers. we obtain
\begin{align*}
(1+t)^Nt^{N+1}F^{(N)}&=(-1)^N\lambda\sum_{i=0}^{N}\left(\lambda^{i-N-1}a_{N-i}(N)t^i\sum_{m\geq0}K_m^{(N+1-i)}(\lambda)\frac{t^m}{m!}
\right)\\
&=(-1)^N\sum_{n\geq0}\left(\sum_{i=0}^{\min(n,N)}\lambda^{i-N}(n)_ia_{N-i}(N)K_{n-i}^{(N+1-i)}(\lambda)
\right)\frac{t^n}{n!}.
\end{align*}
By combining this equation with \eqref{eqAA1}, we complete the
proof.
\end{proof}
\section{Frobenuius-Euler numbers}
Set $F=F(t)=F(t;\lambda,\mu)=\frac{1}{(1+\lambda
t)^{\frac{1}{\lambda}}-\mu}$ ($\lambda\neq0$). By differentiating respect to
$t$, we have
\begin{align}
F^{(1)}&=\frac{-1}{((1+\lambda
t)^{\frac{1}{\lambda}}-\mu)^2}\cdot\frac{(1+\lambda t)^{\frac{1}{\lambda}}}{1+\lambda t}
=\frac{-1}{1+\lambda t}\cdot\frac{(1+\lambda
t)^{\frac{1}{\lambda}}-\mu+\mu}{((1+\lambda t)^{\frac{1}{\lambda}}-\mu)^2}
=\frac{-1}{1+\lambda t}(F+\mu F^2).\label{eqB1}
\end{align}
Now, we let
$$F^{(N)}=\frac{(-1)^N}{(1+\lambda t)^N}\sum_{i=1}^{N+1}b_{i-1}(N;\lambda,\mu)F^i
=\frac{(-1)^N}{(1+\lambda
t)^N}\sum_{i=1}^{N+1}b_{i-1}(N)F^i,$$ for all $N\geq0$, and
$a_i(N)=0$ for all $i\geq N+1$. Note that $F=F^{(0)}=b_0(0)F$,
which implies that $b_0(0)=1$. Also, by \eqref{eqB1}, we have
$b_0(1)=1$ and $b_1(1)=\mu$. For $N+1$, we have
\begin{align*}
F^{(N+1)}&=\frac{d}{dt}\left(\frac{(-1)^{N}}{(1+\lambda t)^{N}}\sum_{i=1}^{N+1}b_{i-1}(N)F^i\right)\\
&=\frac{(-1)^{N+1}\lambda N}{(1+\lambda
t)^{N+1}}\sum_{i=1}^{N+1}b_{i-1}(N)F^i +\frac{(-1)^N}{(1+\lambda
t)^N}\sum_{i=1}^{N+1}ib_{i-1}(N)F^{i-1}F^{(1)},
\end{align*}
which, by \eqref{eqB1}, gives
\begin{align*}
F^{(N+1)}&=\frac{(-1)^{N+1}\lambda N}{(1+\lambda
t)^{N+1}}\sum_{i=1}^{N+1}b_{i-1}(N)F^i
+\frac{(-1)^N\lambda}{(1+\lambda t)^N}\sum_{i=1}^{N+1}ib_{i-1}(N)F^{i-1}\frac{-(F+\mu F^2)}{1+\lambda t}\\
&=\frac{(-1)^{N+1}}{(1+\lambda
t)^{N+1}}\left(\sum_{i=1}^{N+1}(N\lambda+i)b_{i-1}(N)F^i
+\sum_{i=2}^{N+2}(i-1)\mu b_{i-2}(N)F^i\right).
\end{align*}
By assumption, $F^{(N+1)}=\frac{(-1)^{N+1}}{(1+\lambda
t)^{N+1}}\sum_{i=1}^{N+2}b_{i-1}(N+1)F^i$, and, by comparing the
coefficients of $F^i$ on both sides, we obtain the following
recurrence relation
\begin{align}
b_{i-1}(N+1)=(N\lambda+i)b_{i-1}(N)+\mu(i-1)b_{i-2}(N),\quad
i=2,\ldots,N+1\label{eqB2}
\end{align}
with $b_j(N+1)=0$ whenever $j\geq N+2$,
\begin{align}
b_0(N+1)=(N\lambda+1)b_0(N),\quad
b_{N+1}(N+1)=\mu(N+1)b_N(N).
\end{align}

Recalling that $b_0(0)=1$, $b_0(1)=1$ and
$b_1(1)=\mu$, by induction on $N$, we have
\begin{align*}
b_0(N+1)&=(N\lambda+1)((N-1)\lambda+1)\cdots(\lambda+1)=(N\lambda+1|\lambda)_N,\\
b_{N+1}(N+1)&=\mu^{N+1}(N+1)!,
\end{align*}
where $(x\mid\lambda)_n=x(x-\lambda)\cdots(x-(n-1)\lambda)$.

In next lemma, we treat the general case.
\begin{lemma}\label{lemB1}
The coefficients $b_j(N)$, $j=1,2,\ldots,N$, satisfy
$$b_j(N)=j\mu\sum_{i=0}^{N-j}((N-1)\lambda+j+1\mid\lambda)_ib_{j-1}(N-i-1).$$
\end{lemma}
\begin{proof}
By \eqref{eqB2}, we have
\begin{align*}
b_j(N+1)&=j\mu b_{j-1}(N)+(N\lambda+j+1)b_j(N) \\
&=j\mu b_{j-1}(N)
+j\mu(N\lambda+j+1)b_{j-1}(N-1)\\
&+(N\lambda+j+1)((N-1)\lambda+j+1)b_j(N-1).
\end{align*}
By induction and the initial condition $b_j(j)=\mu^{j}j!$, we
obtain
$$b_j(N+1)=j\mu\sum_{i=0}^{N+1-j}(N\lambda+j+1|\lambda)_ib_{j-1}(N-i),$$
as required.
\end{proof}

By Lemma \ref{lemB1}, we can state the following result.
\begin{theorem}\label{thB1}
The function $F=F(t)=\frac{1}{(1+\lambda t)^{\frac{1}{\lambda}}-\mu}$ with
$\lambda\neq0$ satisfies the nonlinear differential equation
$$F^{(N)}=\frac{(-1)^N}{(1+\lambda t)^N}\sum_{i=1}^{N+1}b_{i-1}(N)F^i,$$
where $b_0(N)=((N-1)\lambda+1|\lambda)_{N-1} (N\geq1)$ with $b_0(0)=1$, and
$$b_j(N)=j\mu\sum_{i=0}^{N-j}((N-1)\lambda+j+1|\lambda)_ib_{j-1}(N-i-1),$$
for $j=1,2,\ldots,N$.
\end{theorem}

By considering the proof of Theorem \ref{thB1} and taking  $\lambda\to0$,
we obtain the following result.

\begin{theorem}\label{thB2}
Let $F=F(t)=\frac{1}{e^t-\mu}$. Then
$$F^{(N)}=(-1)^N\sum_{i=1}^{N+1}b_{i-1}(N;\mu)F^i,$$
where $b_0(N;\mu)=b_0(N-1;\mu), b_N(N;\mu)=\mu N b_{N-1}(N-1;\mu)$, and
$b_{i-1}(N;\mu)=ib_{i-1}(N-1;\mu)+\mu(i-1)b_{i-2}(N-1;\mu)$,
for $2\leq i\leq N$
with $b_0(0;\mu)=b_0(1;\mu)=1$ and $b_1(1;\mu)=\mu$.
\end{theorem}

By the recurrence relation of $b_i(N;\mu)$, see Theorem
\ref{thB2}, and by induction on $N$, we obtain $b_0(N;\mu)=1$ and
$b_N(N;\mu)=\mu^NN!$, for all $N\geq0$.

Along the lines of the proof of Lemma \ref{lemB1} and taking $\lambda \to0$,
we obtain that the coefficients $b_j(N;\mu)$, $j=1,2,\ldots,N$,
satisfy
\begin{align}
b_j(N;\mu)=j\mu\sum_{i=0}^{N-j}(j+1)^ib_{j-1}(N-i-1;\mu)\mbox{ with
}b_0(N;\mu)=1.\label{eqlemB2}
\end{align}

{\bf Acknowledgement}: 
The present Research has been conducted by the Research Grant of Kwangwoon University in 2016.


\end{document}